\documentclass{amsart}
\usepackage{latexsym, amssymb,  amsmath}

\newcommand{\R}{{\mathbb R}}
\newcommand{\ds}{{\displaystyle}}
\newcommand{\es}{\mathop{\rm essinf}\limits}
\newcommand{\dist}{{\rm dist}}
\newcommand{\ol}{\overline}
\newcommand{\C}{{\mathcal C}}
\newcommand{\CC}{{\mathfrak C}}
\newcommand{\DD}{{\mathfrak D}}
\newcommand{\E}{{\mathcal  E}}
\newcommand{\M}{{\mathcal M}}
\newcommand{\SSS}{{\mathcal S}}
\newcommand{\ba}{{\mathbf a}}
\newcommand{\bF}{{\mathbf F}}
\newcommand{\bA}{{\mathbf A}}
\newcommand{\bu}{{\mathbf u}}

\newtheorem{thm}{Theorem}[section]

\newtheorem{lem}[thm]{Lemma}
\newtheorem{cor}[thm]{Corollary}
 \newtheorem{defn}[thm]{Definition}
\newtheorem{example}[thm]{Example}
\newtheorem{rem}[thm]{Remark}
\numberwithin{equation}{section}

\begin{document}

\title[Parabolic operators in weighted Lebesgue spaces]{Global gradient estimates in weighted Lebesgue spaces for parabolic operators}

\author[S.-S. Byun]{Sun-Sig Byun}
\address{Department of Mathematical Sciences and Research Institute of Mathematics, Seoul National University,
Seoul 151-747, Korea}
\email{byun@@snu.ac.kr}
\author[D.K. Palagachev]{Dian K. Palagachev}
\address{Dipartimento di Meccanica, Matematica e Management (DMMM), Politecnico di Bari, Via E. Orabona 4, 70125 Bari, Italy}
\email{palaga@poliba.it}
\author[L.G. Softova]{Lubomira G. Softova}
\address{Department of Civil Engineering, Design, Construction Industry and Environment,
  Second University  of Naples, Via Roma 29, 81031 Aversa,   Italy}
\email{luba.softova@unina2.it}

\subjclass[2000]{Primary: 35K20, 35R05; Secondary: 35B65, 35B45, 46E30, 35K40, 42B25, 74E30}

\date{\today}

\keywords{Weighted Lebesgue space; Muckenhoupt weight; Parabolic system; Cauchy--Dirichlet problem; Measurable coefficients; BMO; Gradient estimates; Morrey space; Linear laminates}

\begin{abstract}
We deal with the regularity problem for linear, second order pa\-ra\-bo\-lic equations and systems in divergence form with
measurable data over non-smooth domains,
related to variational problems arising in the modeling of composite materials and in the mechanics of membranes and films of simple non-homogeneous materials which form a linear laminated medium.
Assuming partial BMO smallness of the coefficients
 and Reifenberg flatness of the boundary of the underlying domain, we develop
a Calder\'{o}n--Zygmund type theory for such parabolic operators in the settings of the
weighted Lebesgue spaces. 
As consequence of the main result, we get regularity 
in parabolic Morrey scales for the spatial gradient of the weak solutions to the problems considered.  
\end{abstract}

\maketitle

\section{Introduction}\label{sec1}

The general aim of the present article is to develop a \textit{weighted $L^p$-Calder\'{o}n--Zygmund}  type theory for divergence form, linear parabolic systems with discontinuous coefficients over domains with rough boundary. More precisely, we characterize the regularity of the weak solutions to such systems by deriving global estimates for the spatial gradient in the framework of the weighted Lebesgue spaces, generalizing this way the recent unweighted $L^p$-results of
\textsc{Byun}~\cite{B1,B}, \textsc{Byun, Palagachev} and \textsc{Wang}~\cite{BPW}, \textsc{Byun} and \textsc{Wang}~\cite{BW}, \textsc{Dong} and \textsc{Kim}~\cite{DongKim} and \textsc{Dong}~\cite{Dong}.

Let $\Omega\subset \R^n$ be a bounded domain, $n\geq2,$ and set $Q=\Omega\times(0,T]$ for the cylinder in $\R^{n+1}$ with base $\Omega$ and of height $T.$ We  consider the following
Cauchy-Dirichlet problem
\begin{equation}\label{DP}
\begin{cases}
\ds u_t-  \, D_\alpha(a^{\alpha\beta}(x,t)D_\beta u)= D_\alpha f^\alpha(x,t) &
\ds \text{ in }  Q,\\
\ds  u(x,t)=0 & \text{ on } \partial_P Q,
\end{cases}
\end{equation}
where $\partial_P Q=\partial \Omega \times [0,T] \cup \Omega \times
\{t=0\}$ stands for the parabolic boundary of $Q$ and the summation convention over the repeated indices, running from $1$ to $n,$ is understood.

Suppose that the coefficient matrix $ \ba(x,t)
=\{a^{\alpha\beta}(x,t)\}_{\alpha,\beta=1}^n: \R^{n+1} \to \mathbb{M}^{n\times n}$ is measurable, uniformly bounded and uniformly parabolic, that is,
there exist positive constants $L$ and $\nu$ such that
\begin{equation}\label{eq3}
\begin{cases}
\|a^{\alpha\beta}\|_{L^\infty(\R^{n+1})}\leq L,\\
 a^{\alpha\beta}(x,t)\xi_\alpha\xi_\beta \geq
\nu |\xi|^2 \qquad \forall \  \xi \in \R^n,  \  \text{for almost all}\
(x,t)\in \R^{n+1}.
\end{cases}
\end{equation}
Denote the nonhomogeneous term in \eqref{DP} by $\bF(x,t)
=\big(f^1(x,t),\ldots, f^n(x,t)\big).$ It is well known (cf. \textsc{Byun}~\cite{B}, \textsc{Byun} and \textsc{Wang}~\cite{BW} and the references therein) that if  $\bF\in L^2(Q)$ then the
problem \eqref{DP} has a unique weak solution. Recall  that  a {\it
weak solution} of this problem is a function
$$
u\in C^{0}(0,T; L^2(\Omega))\cap  L^2(0,T; H_0^1(\Omega))
$$
that satisfies
$$
\int_Q u\varphi_t\, dxdt-\int_Q a^{\alpha\beta}D_\beta u D_\alpha
\varphi\, dxdt= \int_Q f^\alpha D_\alpha \varphi\, dx dt
$$
for all $\varphi\in C_0^\infty(Q)$ with $\varphi(x,T)=0.$ Moreover, the
following  $L^2$-estimate holds
\begin{equation}\label{eq5}
\int_Q |Du|^2\, dxdt \leq c \int_Q |\bF|^2\,dxdt,
\end{equation}
where the constant $c$ depends only on $n,$ $L,$ $\nu$ and $T.$

A natural extension of \eqref{eq5} would be the estimate
$$
\int_Q |Du|^p\, dxdt \leq c \int_Q |\bF|^p\,dxdt
$$
with $p>1$ or, more generally,
\begin{equation}\label{eq5'}
\int_Q |Du|^p \omega(x,t)\, dxdt \leq c \int_Q |\bF|^p \omega(x,t)\,dxdt,
\end{equation}
with suitable conditions imposed on the exponent $p$ and the weight $\omega(x,t).$

Indeed, the sole parabolicity of the differential operator considered and boundedness of the underlying domain $\Omega$ are not enough to ensure the validity of \eqref{eq5'} in general. In order to have \eqref{eq5'} for the weak solution to \textit{any} system \eqref{DP}, one has to impose some regularity requirements on the coefficient matrix $\ba,$ some finer geometric assumption on $\partial\Omega$ and suitable conditions on the  weight function $\omega.$

What is our main concern in the present paper is to indicate that set of essentially optimal
hypotheses  on the data of \eqref{DP} which ensure \eqref{eq5'}, and to develop a Calder\'{o}n--Zygmund type theory for the problem under consideration. Namely, taking the nonhomogeneous term $\bF$ in
the weighted Lebesgue space $L^p_\omega(Q)$ (see Sections~\ref{sec2} and \ref{sec3} for the corresponding definitions) we prove that the spatial gradient $Du$ of
the weak solution $u$ to \eqref{DP} belongs to the \textit{same} space $L^p_\omega(Q),$ what is actually the estimate \eqref{eq5'}.

Restricting the value of the exponent $p$ in the range $(2,\infty),$ we consider weights $\omega(x,t)$ belonging to the parabolic Muckenhoupt class $A_{\frac{p}{2}}.$ This is a \textit{necessary and sufficient} restriction ensuring boundedness of the Hardy--Littlewood maximal
operator when acting on the weighted Lebesgue spaces $L^p_\omega.$ For what concerns the coefficients $a^{\alpha\beta}(x,t)$ of the operator considered, we suppose these are \textit{only measurable} with respect to one spatial variable and are averaged in the sense of \textit{small bounded mean oscillation (BMO)} in the remaining space and time variables. This \textit{partially BMO} assumption on the coefficients is quite general and allows \textit{arbitrary} discontinuity in one spatial direction which is often related to problems of linear laminates, while the behaviour with respect to the other directions, including the time, are controlled in terms of small-BMO, such as small multipliers of the Heaviside
step function for instance. It is clear that the cases of continuous, VMO or small-BMO principal
coefficients with respect to \textit{all variables} are particular cases of the situation here considered. Regarding the underlying domain $\Omega,$ we suppose that its \textit{non-smooth boundary}  is \textit{Reifenberg flat} (cf. \textsc{Reifenberg}~\cite{R}) that means $\partial\Omega$
is well approximated by hyperplanes at each point and at each scale.
This is a sort of minimal regularity of the boundary, guaranteeing validity in $\Omega$
of some natural properties of geometric analysis and partial differential equations such as
$W^{1,p}$-extension, nontangential accessibility property, measure density condition,
the Poincar\'e inequality and so on. We refer the reader to
\textsc{David} and \textsc{Toro}~\cite{DT}, \textsc{Kenig} and \textsc{Toro}~\cite{KT}, \textsc{Lemenant, Milakis} and \textsc{Spinolo}~\cite{LMS},
\textsc{Toro}~\cite{Toro}
and the references therein for further details. In particular, a domain which is sufficiently flat in the sense of Reifenberg is also Jones flat. Moreover, domains with $C^1$-smooth or Lipschitz continuous
boundaries with small Lipschitz constant belong to that category, but the class of Reifenberg flat domains extends beyond these common examples and contains domains with rough fractal boundaries such
as the Helge von Koch snowflake.

It is worth noting that the weighted $L^p$-regularity theory here developed is related to important variational problems arising in modeling of deformations in composite materials as fiber-reinforced media or, more generally, in the mechanics of membranes and films of simple non-homogeneous materials which form a linear laminated medium. In particular, a highly twinned elastic or ferroelectric crystal is a typical situation where a laminate appears. The equilibrium equations of such a linear laminate usually have only bounded and measurable coefficients in the
direction of the stratification. We refer the reader to the seminal papers by \textsc{Chipot, Kinderlehrer} and \textsc{Vergara-Caffarelli}~\cite{CKV-C}, \textsc{Li} and \textsc{Vogelius}~\cite{LV} and \textsc{Li} and \textsc{Nirenberg}~\cite{LN} for the general statement of the problem and various issues regarding regularity of solutions in case of piecewise smooth coefficients, and to the more recent works of
\textsc{Elschner, Rehberg} and \textsc{Schmidt}~\cite{ERS}, \textsc{Dong}~\cite{Dong} and
\textsc{Hackl, Heinz} and \textsc{Mielke}~\cite{HHM} for further developments. The
non-smoothness of the underlying Reifenberg flat domain, instead, is related to models of
real-world systems over media with fractal geometry such as blood vessels, the internal structure of lungs, bacteria growth, graphs of stock market data, clouds, semiconductor
devices, etc.

The paper is organized as follows. Section~2  collects some auxiliary results from the harmonic analysis regarding the properties of the Muckenhoupt weights and boundedness of the Hardy--Littlewood maximal operator on the weighted Lebesgue spaces, while in Section~3 we set down the hypotheses on the data of problem \eqref{DP} and state the main result of the paper (Theorem~\ref{MainTh}).
The estimate in the weighted Lebesgue spaces $L^p_\omega$ of the spatial gradient $Du$ of the weak solution to \eqref{DP} is proved in
Section~4. The main analytic tools employed in that proof rely on the Vitali covering lemma, boundedness properties of the Hardy–-Littlewood maximal operator on weighted spaces, power decay estimates of the upper level sets of the spatial gradient and fine properties of the Muckenhoupt weights $\omega(x,t).$ As an outgrowth of the main result, we show
in Section~5 that the Calder\'{o}n--Zygmund property still holds true in the framework of the
parabolic Morrey scales $L^{p,\lambda}$ by showing that $\bF\in L^{p,\lambda}$ yields
$Du\in L^{p,\lambda}.$ A crucial step of our approach here is ensured by an old result of
\textsc{Coifman} and \textsc{Rochberg}~\cite{CR} ensuring that a suitable power of the maximal operator of a characteristic function is an $A_1$-weight. Without essential difficulties, the technique employed in proving regularity of solution to the equation in \eqref{DP} could be extended to the case of systems and, that is why, in the final Section~6 we restrict ourselves to announce only the weighted $L^p$-regularity result for the weak solutions to linear, second order parabolic systems with partially BMO coefficients over Reifenberg flat domains.

To this end, let us note that in the case of \textit{elliptic} equations, weighted $L^p$-regularity results have been proved by \textsc{Mengesha} and \textsc{Phuc}~\cite{MP,MP2} under the small-BMO assumption with respect to \textit{all variables}, and by \textsc{Byun} and \textsc{Palagachev}~\cite{BP,BP2} for equations with \textit{partially} BMO coefficients. To the best of our knowledge, the results here obtained are the first of this kind in the settings of parabolic weighted spaces.

It is worth also noting that regularity of solution in Morrey spaces have been recently derived in \textsc{Softova}~\cite{Sf3} for linear, non divergence form operators with oblique derivative boundary condition by means of estimates for singular integrals of Calder\'on-Zygmund type.
Moreover, the Morrey regularity results from Section~\ref{sec5} could be extended in the more general framework of \textit{generalized} Morrey spaces (see \textsc{Guliyev} and \textsc{Softova}~\cite{GSf} and \textsc{Softova}~\cite{Sf1,Sf5}).

Throughout the paper, the letter $c$ will denote  a universal constant that can be explicitly computed in terms of known quantities such as $n,$ $L,$ $\nu,$ $p,$ $\omega$ and the geometric structure on $Q.$ The exact value of $c$ may vary from one occurrence to another.

\subsection*{Acknowledgements.} S.-S. Byun was supported by the National Research Foundation of Korea(NRF) grant funded by the Korea government(MSIP) (No.2009-0083521). The research of D.K. Palagachev was partially supported by the MIUR-PRIN 2009 project \textit{``Metodi variazionali ed equazioni differenziali non lineari''.}

\section{Weighted Lebesgue spaces in parabolic settings}
\label{sec2}

Our aim is to establish a global weighted estimate of 
Calder\'on-Zygmund type for the weak solution of \eqref{DP} and let us
start with  describing  the properties of  the class of weights
considered. For, we will use the parabolic metric given
in \textsc{Stein}~\cite{St} by the
function
$$
\rho(x,t)=\sqrt{\frac{|x|^2+\sqrt{|x|^4+4|t|^2}}{2} } \qquad\quad
(x,t) \in \R^{n+1}.
$$
It is proved in \textsc{Fabes} and \textsc{Rivi\`ere}~\cite{FR} that $\rho$ defines a metric in $\R^{n+1}$ and
the ``balls'' with respect to it, centered at
$(x,t)$ and of  radius $r>0,$ are  the ellipsoids   $\E = \E_r(x,t)$
$$
\E_r(x,t)=\left\{ (y,\tau) \in  \R^{n+1}: \  \frac{|x-y|^2}{r^2}+
\frac{|t-\tau|^2}{r^4}<1  \right\},
$$
or, in an alternative way,
$$
\E_r(x,t)=\left\{ (y,\tau) \in  \R^{n+1}: \rho(x-y,t-\tau)< r\right\}.
$$

Let $\mu$ be a  nonnegative Borel measure on
$\R^{n+1}$ with the property $\mu(\R^{n+1})>0.$ In the particular case when $\mu$ is the
Lebesgue measure, then $\mu(\E_r)=|\E_r|=c r^{n+2}$ with a
positive constant $c=c(n).$ Let us note that for all points $(x,t),
(y,\tau)\in \R^{n+1}$ and $r>0,$ the collection of such ellipsoids
and the measure that we postulate satisfy the following properties (cf. \textsc{Stein}~\cite{St}):
there exist constants  $c_1$
 and $c_2,$ both greater than 1 and depending on $n,$ such that
\begin{itemize}
\item[{\bf (i)}] \   $\E_r(x,t)\cap \E_r(y,\tau)\not= \emptyset$ implies $\E_r(y,\tau)\subset
\E_{c_1r}(x,t);$
\item[{\bf (ii)}] \   $\mu(\E_{c_1r}(x,t))\leq c_2 \mu(\E_r(x,t));$
\item[{\bf  (iii)}] For each  open set $U$ and $r>0,$ the function $(x,t)\to \mu(\E_r(x,t)\cap U)$ is continuous.
\end{itemize}
Statement {\bf (i)} guarantees the engulfing property crucial in the
Vitali-type covering lemma that we are going to use,  while the
doubling type property {\bf (ii)} of the measure just allows one to
exploit the first statement. In our further considerations we shall
also use the  collection of cylinders  $\C\equiv
\C_r(x,t)=\C_r(x_1,x^\prime,t)$ defined as
\begin{equation}\label{Cr}
\C_r(x,t)=\{(y_1,y',\tau)\in \R^{n+1}:   |x_1-y_1|<r, \
 \rho(x'-y',t-\tau)<r\},
\end{equation}
or, in an alternate  way
\begin{equation}\label{Cr1}
 \C_r(x,t)=\{(y_1,y',\tau)\in \R^{n+1}:   |x_1-y_1|<r, \   \max\{|x'-y'|,\sqrt{|t-\tau|}\} <r
 \}
\end{equation}
with the Lebesgue measure $|\C_r|$ comparable to $r^{n+2},$ and where
$x'=(x_2,\ldots,x_n).$
\begin{rem}\label{rem1}
\em  It is  not
difficult to verify that {\bf (i)}, {\bf (ii)} and  {\bf (iii)}
hold for the collections \eqref{Cr} and \eqref{Cr1}.  In what
follows we will use the equivalence of these structures without explicit
references (cf. \textsc{Stein}~\cite{St}). All definitions given over ellipsoids
$\E_r$ hold also over the cylinders $\C_r.$
\end{rem}

In case of $\R^n$ we shall use also the following collection of
cylinders
$$
\C'_r(x)=\{y\in\R^n: \  |x_1-y_1|<r, |x'-y'|<r\}.
$$
To define the functional spaces to be used in the sequel, we need to
recall the definition and some properties of the Muckenhoupt weights
(cf. \textsc{Garc\'{\i}a-Cuerva} and \textsc{Rubio de Francia}~\cite{GR},
\textsc{Stein}~\cite{St} and \textsc{Torchinsky}~\cite{To}). Let $\M$ denote the Hardy--Littlewood  maximal
operator on $\R^{n+1}$
$$
\M f(x,t)=\sup_{r>0} \frac1{|\E_r(x,t)|}  \int_{\E_r(x,t)}
|f(y,\tau)|\, dy d\tau, \quad  f\in L^1_{\textrm{loc}}(\R^{n+1}).
$$
If $D$ is a bounded domain in $\R^{n+1}$ and $f\in L^1(D),$
then $\M f=\M\bar f,$ where $\bar f$ is the zero extension of $f$ in
the whole space.  It is well known that $\M$ is bounded sub-linear
operator from $L^q$ to itself for all $q>1.$ That is, if  $f\in
L^q(\R^{n+1}), q\in(1,\infty),$ then
\begin{equation}\label{max_ineq}
\int_{\R^{n+1}}|\M f(x,t)|^q \,d\mu(x,t)\leq
c\int_{\R^{n+1}}|f(x,t)|^q\, d\mu(x,t)
\end{equation}
for some positive constant $c=c(q,n),$ where $d\mu=dxdt$ is  the
Lebesgue measure. It turns out that the estimate \eqref{max_ineq}
still holds true when $d\mu=\omega(x,t)dxdt,$ where
$\omega:\R^{n+1}\to \R_+$ is a positive, locally integrable
function, satisfying the following {\it $A_q$-condition}
\begin{equation}\label{Aq}
\left(\frac1{|\E|}\int_\E \omega(x,t)\, dxdt
\right)\left(\frac1{|\E|}\int_\E \omega(x,t)^{-\frac{1}{q-1}}\, dxdt
\right)^{q-1}\leq A< \infty
\end{equation}
for all  $\E$ in $\R^{n+1}.$ It is proved by \textsc{B.~Muckenhoupt} in \cite{Muck} that \eqref{Aq} is a \textit{necessary and sufficient} condition in order \eqref{max_ineq} to hold. By this reason, $\omega$ is commonly called  \textit{Muckenhoupt  weight} lying in the class $A_q,$ and the smallest constant
$A$  for  which   \eqref{Aq} holds is denoted by $[\omega]_q.$
If  $q=1,$ we say that $\omega\in A_1$ when
\begin{equation}\label{A1}
\frac1{|\E|} \int_\E \omega(x,t)\,dxdt \leq A \es_\E \, \omega(x,t).
\end{equation}

There is an alternative way of defining $A_q,$ closely related to the
maximal inequality  \eqref{max_ineq}. For any
nonnegative,  locally  integrable function $f$ and any ellipsoid
$\E,$  the weight $\omega$ belongs to $A_q,$ $1\leq q<\infty,$ if
and only if
\begin{equation}\label{weight2} 
\left( \frac1{|\E|}\int_{\E}  f(x,t)\, dxdt
\right)^q \leq \frac{A}{\omega(\E)} \int_{\E} f^q(x,t) \omega(x,t)\,
dxdt<\infty
\end{equation}
for some   $A=A(q,n)>0,$ where
\begin{equation}\label{weight3}
\omega(\E) = \int_{\E}\omega(x,t)\, dxdt<\infty
\end{equation} is the measure of $\E$ with respect to
$d\mu=\omega(x,t)dxdt.$ The smallest $A$ for which \eqref{weight2}
is valid equals $[\omega]_q.$ It is an immediate consequence of
\eqref{weight2} that whenever $\omega\in A_q$ then it satisfies the doubling property, that is,
\begin{equation}\label{doubling}
\omega(\E_{2r})\leq c(n,q)\omega(\E_r).
\end{equation}
Actually, apply \eqref{weight2} with $\E=\E_{2r}$ and $f=
\chi_{\E_r}$ that gives \eqref{doubling} with $c=[\omega]_q
2^{q(n+2)}.$ The doubling property of $\omega,$ together with
\eqref{weight2}, shows that in the definition \eqref{Aq} we could
have replaced the family of ellipsoids $\{\E_r\}_{r>0}$ by a family
of cylinders $\{\C_r\}_{r>0}$ or other such equivalent families, as
it is noted in Remark~\ref{rem1}.

A noteworthy feature of the $A_q$ classes is that these increase with
$q,$ that is, if $\omega\in A_q,$ then $\omega\in A_p$ whenever
$p\geq q$ and $[\omega]_p\leq [\omega]_q.$

Another  important characteristic  of the Muckenhoupt weights is
the {\it strong doubling property} (see  \textsc{Torchinsky}~\cite[Theorem~IX.2.1]{To} or
\textsc{Mengesha} and \textsc{Phuc}~\cite[Lemma~3.3]{MP}).  Moreover, as proved in \textsc{Stein}~\cite[Section~V.5.3]{St}, for each weight $\omega\in A_q,$
$q>1,$ there exist $\omega_1$ and $\omega_2$ in $A_1$ so that
$\omega=\omega_1\omega_2^{1-q}.$ This,
along with \textsc{Torchinsky}~\cite[Proposition~IX.4.5]{To}, gives  that $\omega$
satisfies the reverse H\"older inequality and {\it reverse doubling
property}. Unifying the both  doubling  conditions, one can observe
that for each   $\E$ and each  measurable  subset $A\subset \E,$
there exist positive constants $c_1$ and $\tau_1\in(0,1)$ such that
\begin{equation}\label{eq8}
\frac1{[\omega]_q}   \left( \frac{|A|}{|\E|}  \right)^q
\leq \frac{\omega(A)}{\omega(\E)}\leq c_1 \left( \frac{|A|}{|\E|}
\right)^{\tau_1},
\end{equation}
where $c_1$ and $\tau_1$ depend on $[\omega]_q,$ $n$ and $q,$ but are
independent of $\E$ and $A.$ Let us note that the lower bound in \eqref{eq8} is the above mentioned \textit{strong doubling property,} while the upper one is the \textit{reverse doubling property.}

Given a measurable and non-negative weight $\omega(x,t),$
the weighted Lebesgue space
$L^q_\omega(\R^{n+1}),$ $q>1,$ is the collection of all measurable functions $f$ for which
\begin{equation}\label{weight4}
\|f\|^q_{L^q_\omega(\R^{n+1})}=\int_{\R^{n+1}}
|f(x,t)|^q\omega(x,t)\,dxdt<\infty.
\end{equation}

As already mentioned above, the famous result of \textsc{Muckenhoupt}~\cite{Muck} states that $\omega\in A_q$ is a \textit{necessary and sufficient condition} ensuring that the Hardy--Littlewood maximal operator maps $L^q_\omega$ into itself (see also \textsc{Torchinsky}~\cite[Theorem~IX.4.1]{To}). Summarizing, we have
\begin{lem}\label{3.2}
Suppose $\omega(x,t)\in A_q,$ $q\in(1,\infty).$ Then there exists a
positive constant $c=c(q,n)$ such that
$$
\frac1c\| f\|_{L^q_\omega(\R^{n+1})}\leq \|\M
f\|_{L^q_\omega(\R^{n+1})}\leq c\|f\|_{L^q_\omega(\R^{n+1})}
$$
whenever $f\in L^q_\omega(\R^{n+1}).$
\end{lem}
\begin{example}
\em
The weight function $ \omega(x,t)=\rho(x,t)^\alpha $ belongs to $ A_q$ with $q\in(1,\infty)$ if and only if $-(n+2)<\alpha < (q-1)(n+2).$
\end{example}

\section{Assumptions and main result}\label{sec3}

For each cylinder $\C_r(y,\tau)=\C_r(y_1,y',\tau)$ and for a fixed
$x_1 \in (y_1-r, y_1+r)$ we set $\C_r^{x_1}(y,\tau)$ to denote
the $x_1$-slice of $\C_r(y,\tau),$ that is,
$$
\C_r^{x_1}(y,\tau)=\{(x',t) \in \R^{n-1} \times \R: (x,t)=(x_1, x',t) \in
\C_r(y,\tau)\}.
$$
Then we define the integral average
$$
\ol \ba_{\C_r^{x_1}(y,\tau)}(x_1)=\frac1{|\C_r^{x_1}(y,\tau)|}
\int_{\C_r^{x_1}(y,\tau)} \ba(x_1,x',t) \, dx'dt.
$$
\begin{defn}\label{def1}
We say that the couple $(\ba, \Omega)$ is $(\delta,R)$-vanishing of codimension $1,$ if the following properties are satisfied:

$\bullet$ For every point $(y,\tau)\in Q$ and for every number $r\in
(0,\frac13 R]$  with
\begin{equation}\label{eq12a}
\dist  (y,\partial \Omega)> \sqrt2 r,
\end{equation}
there exists a coordinate system depending on $(y,\tau)$ and $r,$ whose variables we still denote
by $(x,t)$ so that in this new coordinate system $(y,\tau)$ is the origin and
\begin{equation}\label{eq12}
\frac1{|\C_{r}(0,0)|}\int_{\C_{r}(0,0)}|\ba(x,t)-\ol\ba_{
\C_r^{x_1}(0,0)}(x_1)|^2 \ dxdt\leq \delta^2.
\end{equation}

$\bullet$ For any point $(y,\tau)\in Q$  and for every
number $r\in (0,\frac13 R]$  such that
$$
\dist  (y,\partial \Omega)= \dist (y,x_0)\leq  \sqrt2 r
$$
for some $x_0\in \partial \Omega,$ there exists a coordinate system
depending on $(y,\tau)$ and $r,$ whose variables we still denote by
$(x,t)$ such that in this new coordinate system $(x_0,\tau)$ is the
origin,
\begin{equation}\label{eq13}
\Omega \cap \{ x\in \C'_{3r}(0): x_1> 3r\delta  \}\subset \Omega\cap
\C'_{3r}(0)\subset \Omega \cap \{x\in \C'_{3r}(0): x_1>-3r\delta \}
\end{equation}
and
$$
\frac1{|\C_{3r}(0,0)|}\int_{\C_{3r}(0,0)}|\ba(x,t)
-\ol\ba_{\C_{3r}^{x_1}(0,0)}(x_1)|^2\, dxdt\leq \delta^2.
$$
\end{defn}

We add some comments regarding the above definition. Thanks to the
scaling invariance property, one can take for simplicity $R=1$ or
any other constant bigger than $1.$ On the other hand $\delta$ is a
small positive constant, being invariant under such a scaling
argument. If $\ba$ is $(\delta,R)$-vanishing  of codimension $1,$
then for each point and for each sufficiently small scale, there is
a coordinate system so that the coefficients have small oscillation
in $(x',t)$-variables while these are only measurable
in the $x_1$-variable and therefore may have arbitrary jumps with respect to it. In addition, the boundary of the domain is
$(\delta,R)$-Reifenberg flat (see \textsc{Reifenberg}~\cite{R})  and the coefficients
have a small oscillation along the flat direction $x'$ of the boundary
and are only measurable along the normal direction $x_1.$  The number $\sqrt2 r$ in \eqref{eq12a} is selected for
convenience.  It comes from the reason that we need to take an
enough size of the cylinders in \eqref{eq12}  so that there is a room
to have the rotation of $\C_r(y,\tau)$ in any spatial direction.

We suppose that  the  right-hand side of the equation in  \eqref{DP}
belongs to some weighted Lebesgue space, precisely
$$
|\bF|^2\in L^{\frac{p}{2}}_\omega(Q), \quad \omega\in
A_{\frac{p}{2}}, \quad p\in(2,\infty),
$$
which implies
$$
\bF \in L^{p}_\omega(Q), \quad \omega\in A_{\frac{p}{2}} \subset
A_{p}, \quad p\in(2,\infty).
$$
Then we get from H\"{o}lder's inequality, (\ref{Aq}),
(\ref{weight3}) and (\ref{weight4}) that
\begin{align*}
\|\bF\|^2_{L^2(Q)}& = \int_Q |\bF(x,t)|^2 \omega^{\frac{2}{p}}(x,t) \omega^{-\frac{2}{p}}(x,t)\, dxdt\\
& \leq \left( \int_Q(|\bF(x,t)|^2)^{\frac{p}{2}}\omega(x,t)\,  dxdt  \right)^{\frac{2}{p}} \left(  \int_Q \omega(x,t)^{-\frac{2}{p-2}}\,  dxdt  \right)^{\frac{p-2}{p}}\\
& =
\| |\bF|^2\|_{L^{\frac{p}{2}}_\omega(Q)} |Q|^{\frac{p-2}{p}} \left( \frac1{|Q|}\int_Q\omega(x,t)^{-\frac{2}{p-2}}\, dxdt   \right)^{\frac{p-2}{p}}\\
&  \leq  \| |\bF|^2\|_{L^{\frac{p}{2}}_\omega(Q)} |Q|
\omega(Q)^{-\frac{2}{p}} [\omega]_{\frac{p}{2}}^{\frac{2}{p}}.
\end{align*}
Hence we have
$$
\frac1{c} \|\bF\|^2_{L^2(Q)}\leq \|
|\bF|^2\|_{L_w^{\frac{p}{2}}(Q)}<\infty
$$
with a constant $c=|Q|  \omega(Q)^{-\frac{2}{p}}
[\omega]_{\frac{p}{2}}^{\frac{2}{p}}, $ which ensures the existence
of a unique weak solution $u$ of the equation \eqref{DP} (cf. \textsc{Byun}~\cite{B}, \textsc{Byun} and \textsc{Wang}~\cite{BW}).

We now state the main result of the paper.

\begin{thm}\label{MainTh}
Let $p\in(2,\infty),$ $\omega\in A_{\frac{p}{2}}$ and assume \eqref{eq3}.  Then there
exists a small positive constant $\delta=\delta(n,L,\nu,p,\omega,Q)$
such that if the couple $(\ba,\Omega)$ is $(\delta,R)$-vanishing of codimension $1$ and
 $\bF\in L^p_\omega(Q),$ then the spatial gradient $Du$ of the weak solution $u$ of \eqref{DP} belongs to $L^p_\omega(Q)$ and the following estimate holds
$$
\|Du\|_{L^p_\omega(Q)}\leq c \|\bF\|_{L^p_\omega(Q)}
$$
with a constant $c$ depending on $n,$ $L,$ $\nu,$ $p,$ $\omega$ and $Q.$
\end{thm}

The present work is a natural extension of the previous paper \textsc{Byun, Palagachev} and \textsc{Wang}~\cite{BPW} which deals with the regularity problem for parabolic equations in classical (unweighted, $\omega(x,t)\equiv1$) Lebesgue classes.

Here with a natural parabolic Muckenhoupt weight for the problem \eqref{DP}, we first find a correct version for the weight of the Vitali covering lemma, and verify the hypotheses of this covering lemma from the perturbation results for the unweighted case and comparable relationships between the Lebesgue and the weighted measures. We then apply the covering lemma to derive a weighted power decay estimate of the upper level sets for the Hardy--Littlewood maximal function of the spatial gradient of the weak solution. The required estimate in the main result follows then by the standard procedure of summation over the level sets.

\section{Gradient estimates in $L^p_\omega$}\label{sec4}

Because of the scaling invariance property of the Reifenberg flat domains
(cf. \textsc{Byun, Palagachev} and \textsc{Wang}~\cite[Lemma~5.2]{BPW} for instance), we can take $R=1$ hereafter.
\begin{lem}\label{3.3}
Suppose $\Omega$ is a bounded $(\delta,1)$-Reifenberg flat domain (that is, \eqref{eq13} is verified) and $\omega(x,t)\in A_q,$ $q\in(1,\infty).$
Let  ${\CC}\subset  \DD\subset Q$ be measurable subsets of $Q$
satisfying the following conditions:
\begin{itemize}
\item[$\bullet$] there exists
$\varepsilon\in(0,1)$ such that for each $(y,\tau)\in Q$
\begin{equation}\label{eq3.3}
\omega(\CC\cap \C_1(y,\tau))< \epsilon \omega(\C_1(y,\tau));
\end{equation}
\item[$\bullet$] for each $(y,\tau)\in Q$ and $r>0$
\begin{equation}\label{eq3.4}
\omega(\CC\cap \C_r(y,\tau))\geq \varepsilon
\omega(\C_r(y,\tau)) \quad   \text{implies} \quad Q\cap \C_r(y,\tau)\subset \DD.
\end{equation}
\end{itemize}
Then
\begin{equation}\label{eq3.3a}
 \omega(\CC) \leq \varepsilon [\omega]_q^2 \left( \frac{10\sqrt2}{1-\delta} \right)^{q(n+2)}
 \omega(\DD).
\end{equation}
\end{lem}
\begin{proof}
Fix $(y,\tau)\in Q$ and for each $r>0$ define the function
$$
\Theta(r)=\frac{\omega(\CC\cap \C_r(y,\tau))}{\omega(\C_r(y,\tau))}.
$$
We have $\Theta\in C^0(0,\infty),$
$\Theta(1)<\varepsilon$ according
to \eqref{eq3.3} and $ \Theta(0)=\lim_{r\to 0_+}\Theta(r)=1$ by the
Lebesgue Differentiation Theorem. Therefore, for almost all
$(y,\tau)\in \CC,$ there exists $r_{(y,\tau)}\in(0,1)$ such that
$\Theta(r_{(y,\tau)})=\varepsilon$ and $\Theta(r)<\varepsilon$ for
all $r>r_{(y,\tau)}. $

Define the family of cylinders $\{\C_{r_{(y,\tau)}}(y,\tau)\}_{(y,\tau)\in \CC}$ which forms an open covering of $\CC.$ By the Vitali lemma
(cf. \textsc{Stein}~\cite[Lemma~I.3.1]{St}), there exists a disjoint sub-collection
  $\{\C_{r_i}(y_i,\tau_i)\}_{i\geq 1} $ with $r_i=r_{(y_i,\tau_i)}\in(0,1),$ $(y_i,\tau_i)\in \CC$ such that $\Theta(r_i)=\varepsilon$ and
$$
\sum_{i\geq 1} |\C_{r_i}(y_i,\tau_i)|\geq c |\CC|,\quad
\CC\subset \ds\bigcup_{i\geq 1}\C_{5r_i}(y_i,\tau_i),\quad
\CC = \bigcup_{i\geq 1}\big(\CC\cap
\C_{5r_i}(y_i,\tau_i)\big)
$$
for some $c=c(n)>0.$

Since $\Theta(5r_i)<\varepsilon,$ we have  by \eqref{eq8}
\begin{align*}
&\omega(\CC\cap \C_{5r_i}(y_i,\tau_i))<\varepsilon \omega(\C_{5r_i}(y_i,\tau_i))\\
&\leq \varepsilon [\omega]_q \left(  \frac{|\C_{5r_i}(y_i,\tau_i)|}{|\C_{r_i}(y_i,\tau_i)|} \right)^q
\omega(\C_{r_i}(y_i,\tau_i))\\
&=\varepsilon[\omega]_q5^{q(n+2)}\omega(\C_{r_i}(y_i,\tau_i)).
\end{align*}

In order to employ \eqref{eq3.4}, we have to estimate the ratio
$\frac{\omega(\C_{r_i}(y_i,\tau_i))}{\omega(Q\cap \C_{r_i}(y_i,t_i))}.$ For, making use of the bound
$$
\sup_{0<r<1}\sup_{(y,\tau)\in Q}\frac{|\C_r(y,\tau)|}{|Q\cap \C_r(y,\tau)|}\leq \left(
\frac{2\sqrt 2}{1-\delta} \right)^{n+2},
$$
obtained in \textsc{Byun, Palagachev} and \textsc{Wang}~\cite{BPW} or \textsc{Byun} and \textsc{Wang}~\cite{BW2}, and the doubling condition
\eqref{eq8}, we get
$$
\omega(\C_{r_i} (y_i,\tau_i)  )\leq [\omega]_q \left(
\frac{2\sqrt 2}{1-\delta} \right)^{q(n+2)} \omega(Q\cap \C_{r_i}(y_i,\tau_i)).
$$
Now we have
\begin{align*}
\omega(\CC)& \leq \omega\big(\bigcup_{i\geq 1}\big(\CC\cap \C_{5r_i}(y_i,\tau_i)\big)\big)\\
&\leq \sum_{i\geq 1}\omega(\CC\cap\C_{5r_i}(y_i,\tau_i)) < \varepsilon \sum_{i\geq 1}\omega(\C_{5r_i}(y_i,\tau_i) )\\
&\leq \varepsilon
[\omega]_q5^{q(n+2)}\sum_{i\geq 1} \omega( \C_{r_i}(y_i,\tau_i))\\
&\leq \varepsilon
[\omega]^2_q \left( \frac{10\sqrt2}{1-\delta}  \right)^{q(n+2)} \sum_{i\geq 1}
\omega(Q\cap\C_{r_i}(y_i,\tau_i)).
\end{align*}
Having in mind that $\{\C_{r_i}(y_i,\tau_i)\}$ are mutually
disjoint, $\Theta(r_i)=\varepsilon$ and \eqref{eq3.4},  we get
\begin{align*}
\omega(\CC)&\leq \varepsilon [\omega]_q^2\left( \frac{10\sqrt2}{1-\delta}  \right)^{q(n+2)}
\omega\Big(\bigcup_{i\geq 1}Q\cap \C_{r_i}(y_i,\tau_i)\Big)\\
&\leq \varepsilon[\omega]_q^2 \left( \frac{10\sqrt2}{1-\delta}  \right)^{q(n+2)} \omega(\DD).
\end{align*}
\end{proof}

In the following we recall an approximation lemma obtained for the
unweighted spaces in \textsc{Dong} and \textsc{Kim}~\cite[Corollary~8.4]{DongKim}, \textsc{Byun, Palagachev} and \textsc{Wang}~\cite[Lemma~5.3]{BPW} and \textsc{Byun} and \textsc{Wang}~\cite[Lemma~5.5]{BW2}.
\begin{lem}\label{3.4}
Assume \eqref{eq3} and let $u$ be a weak solution of \eqref{DP}. Then there is a constant $\lambda_1=\lambda_1(\nu, n)>1$
such that for each $\varepsilon\in(0,1)$ there exists $\delta=\delta(\varepsilon)>0$ such that if $\ba$ is $(\delta,1)$-vanishing of codimension $1$ and if $\C_r(y,\tau)$ satisfies
$$
|\{(x,t)\in Q\colon  \M(|Du|^2)>\lambda_1^2  \}\cap \C_r(y,\tau) |\geq
\varepsilon |\C_r(y,\tau)|,
$$
then we have
$$
Q\cap \C_r(y,\tau)\subset \{(x,t)\in Q\colon  \M(|Du|^2)>1  \}\cup \{ \M(|\bF|^2)>\delta^2 \}.
$$
\end{lem}

We need now to establish  a weighted version of the  lemma cited above.   For this goal, for any weak solution $u$ of \eqref{DP} we set
\begin{equation}\label{eq3.9}
\CC=\{(x,t)\in Q\colon   \M(|Du|^2)>\lambda_1^2  \}
\end{equation}
and
\begin{equation}\label{eq3.10}
\DD=\{(x,t)\in Q\colon   \M(|Du|^2)>1  \} \cup \{ \M(|\bF|^2)>\delta^2  \}
\end{equation}
with $\lambda_1$ and $\delta$ as in Lemma~\ref{3.4}. The next assertion shows that the assumption \eqref{eq3.4} holds for the such defined sets $\CC$ and $\DD.$
\begin{lem}\label{3.5}
Let $\omega\in A_q, $ $q\in (1,\infty).$ Assume that  $\ba$ is $(\delta,1)$-vanishing of
 codimension 1 and for each $r>0$  and  almost all $(y,\tau)\in Q,$   $\C_r(y,\tau)$ satisfies
$$
\Theta(r)=\frac{\omega(\CC\cap \C_r(y,\tau))}{ \omega(\C_r(y,\tau))}\geq \varepsilon.
$$
 Then we have $Q\cap\C_r(y,\tau)\subset \DD.$
\end{lem}
\begin{proof}
The reverse  doubling property of $\omega$ (the upper bound in \eqref{eq8})  gives that
$$
\varepsilon\leq \frac{\omega(\CC\cap
\C_r(y,\tau))}{\omega(\C_r(y,\tau))}\leq c_1 \left(
 \frac{|\CC\cap \C_r(y,\tau)|}{|\C_r(y,\tau)|}\right)^{\tau_1}.
$$
  Hence
$$
|\CC\cap \C_r(y,\tau)|\geq \left( \frac{\varepsilon}{c_1}
\right)^{\frac{1}{\tau_1}} |\C_r(y,\tau)|.
$$
The assertion holds after applying Lemma~\ref{3.4} with $\varepsilon$ replaced by $\left( \frac{\varepsilon}{c_1}
\right)^{\frac{1}{\tau_1}}.$
\end{proof}

We are going to derive now the power decay estimate  of the upper
level set $\CC$ with respect to $A_{\frac{p}{2}}$-weights.
\begin{lem}\label{3.6}
Under the assumptions of  Lemma~\ref{3.5} we suppose additionally that
\begin{equation}\label{eq3.14}
\Theta(1)=\frac{\omega(\CC\cap \C_1(y,\tau))}{\omega(\C_1(y,\tau))}<\varepsilon
\end{equation}
with $\CC$ as in \eqref{eq3.9}.
Then for each $ k=1,2,\ldots,$  we have
\begin{align}\label{eq3.15}
\nonumber
&\omega\left(\{(x,t)\in Q\colon  \M(|Du|^2)>\lambda_1^{2k}    \}\right)\leq
\varepsilon_1^k\omega\left(\{(x,t)\in Q\colon  \M(|Du|^2)>1   \}\right)\\
&\qquad\qquad\qquad\qquad\qquad +\sum_{i=1}^k \varepsilon_1^i \omega\left(\left\{(x,t)\in Q\colon
\M(|\bF|^2)>\delta^2\lambda_1^{2(k-i)} \right\}\right)
\end{align}
where $\varepsilon_1=\varepsilon[\omega]^2_{\frac{p}{2}} \left(
\frac{10\sqrt2}{1-\delta}  \right)^{\frac{p}2(n+2)}.$
\end{lem}
\begin{proof}
Lemma~\ref{3.5} and condition \eqref{eq3.14} ensure
the validity of the hypotheses of  Lemma~\ref{3.3} for  the sets
\eqref{eq3.9} and \eqref{eq3.10}. Thus, we get  by \eqref{eq3.3a}
\begin{align*}
\omega\left(\{(x,t)\in Q\colon  \M(|Du|^2)>\lambda_1^2\}\right)
\leq &\ \varepsilon_1\omega\left(\{{(x,t)\in Q\colon  \M(|Du|^2)>1 }\} \right)\\
& + \varepsilon_1\omega\left(\{(x,t)\in Q\colon  \M(|\bF|^2)>\delta^2\}
\right),
\end{align*}
where $\varepsilon_1= \varepsilon [\omega]_{\frac{p}{2}}
\left(\frac{10\sqrt2}{1-\delta} \right)^{\frac{p}2(n+2)}.$

The last inequality is exactly \eqref{eq3.15} with $k=1.$ Further, we
proceed with the proof  by induction,  as it is done in
\textsc{Byun}~\cite[Corollary~4.15]{B1}.  Suppose that \eqref{eq3.15} holds true for each weak solution of \eqref{DP}  and for some $k>1.$ Define the functions $u_1=\frac{u}{\lambda_1}$ and
$\bF_1=\frac{\bF}{\lambda_1}.$  It is easy to see that $u_1$ is a
weak solution to the problem \eqref{DP} with a right-hand side
$\bF_1.$ Hence, \eqref{eq3.14} and Lemma~\ref{3.5}  hold  
with the sets $\CC$ and $\DD$ corresponding to $u_1$ as defined by
\eqref{eq3.9} and \eqref{eq3.10}
and according to \eqref{eq3.15}, the inductive assumption holds true for $u_1$ with the same $k>1.$
 The definition of $u_1$ ensures the inductive passage from $k$ to $k+1$ for $u.$ Namely,
\begin{align*}
&\omega\left(\left\{(x,t)\in Q\colon  \M(|Du|^2)>\lambda_1^{2(k+1)}  \right\}\right)\\
&\qquad\qquad =\omega\left(\{(x,t)\in Q\colon  \M(|Du_1|^2)>\lambda_1^{2k}  \}\right)\\
&\qquad\qquad\leq \varepsilon_1^k \omega\left(\{(x,t)\in Q\colon  \M(|Du_1|^2)>1   \}\right)\\
&\qquad\qquad\quad +\sum_{i=1}^k \varepsilon_1^i
\omega\left(\left\{(x,t)\in Q\colon \M(|\bF_1|^2)>\delta^2\lambda_1^{2(k-i)} \right\}\right)\\
&\qquad\qquad = \varepsilon_1^k \omega\left(\{(x,t)\in Q\colon  \M(|Du|^2)>\lambda_1^2 \}\right)\\
&\qquad\qquad\quad +\sum_{i=1}^k \varepsilon_1^i \omega\left(\left\{(x,t)\in Q\colon
\M(|\bF|^2)>\delta^2\lambda_1^{2(k-i)}\lambda^2 \right\}\right)\\
&\qquad\qquad\leq \varepsilon_1^{k+1}\omega\left(\{(x,t)\in Q\colon  \M(|Du|^2)>1   \}\right)\\
&\qquad\qquad\quad +\sum_{i=1}^{k+1} \varepsilon_1^i \omega\left(\left\{(x,t)\in Q\colon
\M(|\bF|^2)>\delta^2\lambda_1^{2(k+1-i)} \right\}\right).
\end{align*}
\end{proof}

The next result follows directly from \textsc{Mengesha} and \textsc{Phuc}~\cite[Lemma~3.7]{MP}.
\begin{lem}\label{3.7}
Let $h\in L^1(Q)$  be a nonnegative   function,  $\omega$ be an $ A_q$-weight,  $q\in(1,\infty)$ and  $\theta>0, \Lambda>1$ be constants. Then $h\in L^q_\omega(Q)$ if and only if
$$
\SSS:=\sum_{k\geq 1}\Lambda^{kq}\omega(\{ (x,t)\in Q\colon  h(x,t)>\theta\Lambda^k  \})<\infty.
$$
Moreover,
$$
c^{-1}\SSS\leq \|h\|^q_{L^q_\omega(Q)}\leq c (\omega(Q)+\SSS),
$$
where $c=c(\theta,\Lambda,q).$
\end{lem}

We are in a position now to prove Theorem~\ref{MainTh}.
\begin{proof}
Since $\omega\in A_{\frac{p}{2}}$ then  $\omega\in A_p$ with
$[\omega]_p\leq [\omega]_{\frac{p}{2}}.$ Recall that $\bF\in
L^p_\omega(Q)$ and from the scaling  invariance property of
\eqref{DP} under a normalization, we assume that
$\|\bF\|_{L^p_\omega(Q)}\leq \delta$ with $\delta>0$ small enough.
Hence
\begin{equation}\label{eq3.16}
\|\bF\|^2_{L^p_\omega(Q)}  = \|
|\bF|^2\|_{L_\omega^{\frac{p}{2}}(Q)}\leq \delta^2.
\end{equation}
Then we need to prove boundedness of the norm of the gradient
$\|Du\|_{L^p_\omega(Q)}.$ Because of the properties of the maximal
function (see Lemma~\ref{3.2}), it is enough to get
$$
\|\M(|Du|^2)\|_{L^{\frac{p}{2}}_\omega(Q)}\leq c.
$$
For this goal, we apply Lemma~\ref{3.7} with $h=\M(|Du|^2),$
$\Lambda=\lambda_1^2,$ $q=\frac{p}{2},$ $\theta=1.$ By the reverse
doubling property \eqref{eq8}, we have
$$
\frac{\omega(\CC\cap \C_1(y,\tau))}{\omega(\C_1(y,\tau))}\leq
c\left( \frac{|\CC\cap \C_1(y,\tau)|}{|\C_1(y,\tau)|}
\right)^{\tau_1}.
$$
To estimate the right-hand side, we note that 
\begin{align*}
\frac{|\CC\cap \C_1(y,\tau)|}{|\C_1(y,\tau)|}\leq &\ c|\CC|=
c| \{(x,t)\in Q\colon  \M(|Du|^2)>\lambda_1^2   \} |\\
\nonumber
\leq &\ c \int_Q\M(|Du|^2)\,dxdt\leq c \int_Q |Du|^2\, dxdt\\
\leq &\ c \int_Q|\bF(x,t)|^2\,dxdt\leq \frac{c |Q|
}{\omega(Q)^{\frac{2}{p}}}\left( \int_Q |\bF(x,t)|^{2\frac{p}2}
\omega(x,t)\,dxdt  \right)^{\frac2p}\leq c \delta^2,
\end{align*}
for almost all $(y,\tau)\in \CC,$ where we have used \eqref{weight2}.    Taking $\delta$ small enough, we get by \eqref{eq8} that
$$
\Theta(1)=\frac{\omega(\CC\cap \C_1(y,\tau))}{\omega(\C_1(y,\tau))}\leq
c\delta^{2\tau_1} <\varepsilon,
$$
which ensures the validity of \eqref{eq3.14}. Therefore
Lemma~\ref{3.6} gives
\begin{align*}
\SSS:=&\ \sum_{k\geq 1} \lambda_1^{2k\frac{p}{2}}   \omega\left(\{(x,t)\in Q\colon
 \M(|Du|^2)>\lambda_1^{2k}\}\right)\\
\leq &\ \sum_{k\geq1} \lambda_1^{kp}\varepsilon_1^k  \omega\left(\{(x,t)\in Q\colon  \M(|Du|^2)>1\}\right)\\
&\ + \sum_{k\geq 1}\sum_{i=1}^k \lambda_1^{kp}\varepsilon_1^i
\omega\left(\left\{(x,t)\in Q\colon  \M(|\bF|^2)>\delta^2\lambda_1^{2(k-i)}\right\}\right)\\
\leq&\ \sum_{k\geq 1}\big(\lambda_1^p\varepsilon_1  \big)^k
\omega(Q)\\
&\ + \sum_{i\geq 1} \big(\lambda_1^p\varepsilon_1  \big)^i \sum_{k\geq i}
\lambda_1^{p(k-i)} \omega\left(\left\{(x,t)\in Q\colon  \M(|\bF|^2)>\delta^2\lambda_1^{2(k-i)}\right\}\right)\\
\leq&\ \omega(Q) \sum_{k\geq 1}\big(\lambda_1^p\varepsilon_1  \big)^k+ \SSS'.
\end{align*}
Let us note that
\begin{align*}
&\sum_{k\geq i} \lambda_1^{p(k-i)}\omega \left(\left\{ (x,t)\in Q\colon  \M(|\bF|^2)>\delta^2\lambda_1^{2(k-i)} \right\}\right)\\
&\qquad=
\sum_{k\geq i} \big( \lambda_1^{2(k-i)}\big)^{\frac{p}2}\omega \left(\left\{ (x,t)\in Q\colon   \M\left(
\frac{|\bF|^2}{\delta^2}\right)>\lambda_1^{2(k-i)} \right\}\right)\\
&\qquad \leq C\left\|
\frac{|\bF|^2}{\delta^2}\right\|^{\frac{p}{2}}_{L^{\frac{p}{2}}_\omega(Q)}=\frac{c}{\delta^p}
\big\| |\bF|^2 \big\|^{\frac{p}{2}}_{L^{\frac{p}{2}}_\omega(Q)}\leq
c,
\end{align*}
where \eqref{eq3.16} has been used in the last estimate. Hence
$\SSS'\leq c \sum_{i\geq 1} (\lambda_1^p\varepsilon_1)^i.$

Taking $\varepsilon$ small enough in a way that
$\lambda_1^p\varepsilon_1<1,$ and consequently also $\delta,$ we get $\SSS<\infty$ which gives
$$
\|\M(|Du|^2)\|_{L^{\frac{p}{2}}_\omega(Q)}\leq
c(\omega(Q)+\SSS)<\infty.
$$

This way, Lemmas~\ref{3.2} and \ref{3.7} imply
$$
\|Du\|_{L^p_\omega(Q)}\leq c \left(\|\bF\|_{L^p_\omega(Q)}+\omega(Q)\right)
$$
and the desired estimate which completes the proof of Theorem~\ref{MainTh}
follows by the Banach inverse mapping theorem.
\end{proof}

\section{Morrey regularity of the gradient}\label{sec5}

A direct consequence of Theorem \ref{MainTh} is the Morrey
regularity of the spatial gradient of the weak solution to the problem \eqref{DP}. Recall
that the Morrey spaces $L^{q,\lambda}(Q)$ with $q>1$ and $\lambda\in(0,n+2)$ consist of all measurable  functions
 $f\in L^q(Q)$  for which the following norm is finite
$$
\|f\|_{L^{q,\lambda}(Q)}=\left(
\sup_{\underset{\footnotesize 0<r<\text{diam\,}Q}{(y,\tau)\in Q}}\frac1{r^\lambda}\int_{\E_r(y,\tau)\cap Q} |f(x,t)|^q dxdt
\right)^{1/q}<\infty,
$$
where $\E_r(y,\tau)$ is any ellipsoid with radius $r$ and
centered at $(y,\tau)\in Q.$

\begin{thm}\label{th2}
Under the assumptions of Theorem~\ref{MainTh}, suppose in addition that $\bF\in L^{p,\lambda}$ with $p>2$ and $\lambda\in (0, n+2).$ Then
there
exists a small positive constant $\delta=\delta(n,L,\nu,p,\lambda,Q)$
such that if the couple $(\ba,\Omega)$ is $(\delta,R)$-vanishing of codimension $1,$ then the spatial gradient $Du$ of the weak solution $u$ to the problem  \eqref{DP} belongs to $L^{p,\lambda}(Q)$ and satisfies
the estimate
\begin{equation}\label{morrey}
\|Du\|_{L^{p,\lambda}(Q)}\leq c\|\bF\|_{L^{p,\lambda}(Q)}
\end{equation}
with a constant $c$  independent of $u$ and $\bF.$
\end{thm}
\begin{proof}
Suppose  that  $\bF\colon Q\to \R^n$ is  extended as zero to the whole  $\R^{n+1}.$ Fix a point $(x_0,t_0)\in Q,$ $r>0,$ and consider the ellipsoid  $\E_r(x_0,t_0)$  with a characteristic function $\chi_{\E_r(x_0,t_0)}$   and maximal function $\M \chi_{\E_r(x_0,t_0)}(x,t).$  It is proved by \textsc{Coifman} and  \textsc{Rochberg}  (see \cite[Proposition~2]{CR} or \textsc{Torchinsky}~\cite[Proposition~IX.3.3]{To}) that
$$
\left( \M \chi_{\E_{(x_0,t_0)}} \right)^\sigma\in A_1\quad \text{ when }\quad 0\leq \sigma<1.
$$
Hence, by the definition \eqref{A1} of an $A_1$-weight,  we get
$$
\frac1{|\E_r(x_0,t_0)|} \int_{\E_r(x_0,t_0)} \Big(\M\chi_{\E_r(x_0,t_0)}(x,t)  \Big)^\sigma\, dxdt
\leq A\es_{\E_r(x_0,t_0)}\Big( \M\chi_{\E_r(x_0,t_0)}\Big)^\sigma.
$$
Because of the increasing property of the $A_q$-classes we have
$\left(\M\chi_{\E_r(x_0,t_0)}\right)^\sigma\in A_{\frac{p}{2}}$ for
each $p>2$ with a bound $\big[( \M\chi_{\E_r(x_0,t_0)})^\sigma
\big]_{\frac{p}{2}}\leq A$ depending only on $n,$ $p$ and $\sigma.$
Therefore, applying the result of Theorem~\ref{MainTh}, we obtain
\begin{align*}
I&\ =\int_{\E_r(x_0,t_0)\cap Q} |Du|^p dxdt= \int_Q |Du|^p \big( \chi_{\E_r(x_0,t_0)} \big)^\sigma dxdt\\
&\ \leq \int_Q|Du|^p \big( \M\chi_{\E_r(x_0,t_0)} \big)^\sigma dxdt\leq c \int_Q|\bF(x,t)|^p \big( \M\chi_{\E_r(x_0,t_0)} \big)^\sigma dxdt\\
&\ =c \int_{\R^{n+1}}|\bF(x,t)|^p \big( \M\chi_{\E_r(x_0,t_0)}
\big)^\sigma dxdt.
\end{align*}
Employing the dyadic decomposition of $\R^{n+1}$ related to $\E_r(x_0,t_0),$
$$
\R^{n+1}=\E_{2r}(x_0,t_0)\cup \left( \bigcup_{k=1}^\infty \E_{2^{k+1}r}(x_0,t_0) \setminus \E_{2^kr} (x_0,t_0)  \right),
$$
the last bound becomes
\begin{align*}
I\leq&\ c\left( \int_{\E_{2r}(x_0,t_0)}|\bF(x,t)|^p   \big( \M\chi_{\E_r(x_0,t_0)} \big)^\sigma dxdt\right.\\
&\qquad + \left.\sum_{k=1}^\infty \int_{\E_{2^{k+1}r}(x_0,t_0)\setminus \E_{2^kr}(x_0,t_0)} |\bF(x,t)|^p  \big( \M\chi_{\E_r(x_0,t_0)} \big)^\sigma dxdt\right)\\
=&\ I_0+\sum_{k=1}^\infty I_k.
\end{align*}

Let us estimate now the maximal function of $\chi_{\E_r(x_0,t_0)}.$ It follows by the definition that
\begin{align*}
\M\chi_{\E_r(x_0,t_0)}(y,\tau)&= \sup_{\E_{\bar r}(y,\tau)}\frac1{|\E_{\bar r}(y,\tau)|} \int_{\E_{\bar r}(y,\tau)} \chi_{\E_r(x_0,t_0)}(x,t)\,dxdt\\
&=\sup_{\E_{\bar r}(y,\tau)} \frac{|\E_{\bar r}(y,\tau)\cap
\E_r(x_0,t_0)|}{| \E_{\bar r}(y,\tau) |}
 =\sup_{\bar r}\frac{r^{n+2}}{\bar r^{n+2}},
\end{align*}
where  $\E_{\bar r}(y,\tau)$ is an arbitrary ellipsoid centered at some point $(y,\tau)\in \R^{n+1}.$

If $(y,\tau)\in \E_r(x_0,t_0),$ then  $\M\chi_{\E_r(x_0,t_0)}(y,\tau)=1.$
On the other hand, if  $(y,\tau)\in \E_{2r}(x_0,t_0)\setminus
\E_r(x_0,t_0),$ then
$$
\M\chi_{\E_r(x_0,t_0)}(y,\tau)= \frac{r^{n+2}}{(2r)^{n+2}}=\frac1{2^{n+2}}<1.
$$
Let $(y,\tau)\in  \E_{2^{k+1}r}(x_0,t_0)\setminus \E_{2^kr}(x_0,t_0).$
We have
\begin{align*}
2^{k-1}r\leq&\ 2^kr-r\leq \rho(y-x_0, \tau-t_0)-r \leq \bar r\\
 \leq&\ \rho(y-x_0,\tau-t_0)+r\leq 2^{k+1}r+r,
\end{align*}
and the maximal function  majorizes  as
$$
\M\chi_{\E_r(x_0,t_0)}(y,\tau)\leq \frac{r^{n+2}}{(2^{k-1}r)^{n+2}}=\frac1{2^{(k-1)(n+2)}}.
$$

We are in a position now to estimate the terms $I_k,$ $k=0,1,\ldots.$
Namely,
\begin{align*}
I_0\leq&\ \int_{\E_{2r}(x_0,t_0)} |\bF(x,t)|^p\, dxdt \leq C(n) r^\lambda\|\bF\|^p_{L^{p,\lambda}(Q)},\\
I_k\leq&\ \frac1{2^{(k-1)(n+2)\sigma}} \int_{\E_{2^{k+1}r\setminus \E_{2^kr}}(x_0,t_0)}
|\bF(x,t)|^p\,dxdt\\
\leq&\  \frac1{2^{(k-1)(n+2)\sigma}} \int_{\E_{2^{k+1}r}(x_0,t_0)}
|\bF(x,t)|^p\,dxdt\\
=&\ \frac{(2^{k+1}r)^\lambda}{2^{(k-1)(n+2)\sigma}}\frac1{(2^{k+1}r)^\lambda}
\int_{\E_{2^{k+1}r} }
|\bF(x,t)|^p\,dxdt\\
\leq&\  C(n,\sigma,\lambda)  2^{k(\lambda-(n+2)\sigma)} r^\lambda\|\bF\|^p_{L^{p,\lambda}(Q)},
\end{align*}
which leads to
$$
\int_{\E_r(x_0,t_0)\cap Q} |Du|^p dxdt\leq C\left(1+\sum_{k=1}^\infty 2^{\lambda-(n+2)\sigma}  \right) r^\lambda \|\bF\|^p_{L^{p,\lambda}(Q)}.
$$

At this point we choose $\sigma\in\left(\frac{\lambda}{n+2},1\right)$ in order to ensure convergence of the series above. To get the
desired estimate \eqref{morrey}, it remains to divide the both sides by $r^\lambda$ and to take supremum with respect to
$0<r<\text{diam\,}Q$ and $(x_0,t_0)\in Q.$
\end{proof}

\section{Linear parabolic  systems in divergence form}

The previous results could be easily extended to the case of
nonhomogeneous parabolic systems in divergence form
\begin{equation}\label{system}
\begin{cases}
u_t^i-D_\alpha\big( a^{\alpha\beta}_{ij}(x,t)D_\beta u^j  \big)=D_\alpha f_\alpha^i(x,t)
& \text{ in } Q,\\
u^i(x,t)=0 & \text{ on } \partial_P Q,
\end{cases}
\end{equation}
for $i=1,\ldots,m.$

The tensor matrix of the coefficients
$$\bA=\{a^{\alpha\beta}_{ij}\}\colon \R^{n+1}\to \R^{mn\times mn}$$  is
assumed to be uniformly bounded and uniformly parabolic, namely, we
suppose that there exists positive constants $L$ and $\nu$ such
that
\begin{equation}\label{parab1}
\|\bA\|_{L^\infty(\R^{n+1}, \R^{mn\times mn})}\leq L,\qquad
a^{\alpha\beta}_{ij}(x,t)\xi^i_\alpha\xi^j_\beta\geq \nu |\xi|^2
\end{equation}
for all matrices $\xi\in \M^{m\times n}$ and for almost every
$(x,t)\in \R^{n+1}.$

When the nonhomogeneous term
$\bF=\{f_\alpha^i\}$ belongs to $L^2(Q,\R^{mn}),$ the Cauchy--Dirichlet
problem \eqref{system} has a unique weak solution $\bu=(u_1,\ldots,
u_m)$ with the standard $L^2$-estimate
$$
\|D\bu\|_{L^2(Q,\R^{nm})}\leq c\|\bF\|_{L^2(Q,\R^{nm})},
$$
where $c$ is a positive constant depending only on $n,$ $m,$ $L,$ $\nu$ and $|Q|.$ In particular, the weak solution of \eqref{system} belongs to
$$
H^{\frac{1}{2}}(0,T;L^2(\Omega, \R^m))\cap L^2(0,T;
H_0^1(\Omega,\R^m)),
$$
and satisfies the estimate
$$
\|\bu\|_{H^{\frac{1}{2}}(0,T;L^2(\Omega, \R^m))\cap L^2(0,T;
H_0^1(\Omega,\R^m))}+\|D\bu\|_{L^2(Q, \R^{mn})}\leq c
\|\bF\|_{L^2(Q,\R^{mn})},
$$
where the constant $c$ is independent of $\bu$ and $\bF.$

The proofs given in Sections~\ref{sec4} and \ref{sec5} apply also to the weak solutions of the system \eqref{system}. That is why, we shall restrict ourselves only to announce the corresponding regularity results.

\begin{thm}\label{th3} 
Assume \eqref{parab1} and let $p\in(2,\infty)$ and  $\omega\in
A_{\frac{p}{2}}.$ There exists a small positive constant
$\delta=\delta(n,m,L,\nu,p,\omega)$ such that if the couple $(\bA,\Omega)$ is
$(\delta,R)$-vanishing of codimension 1 and $\bF\in
L^p_\omega(Q,\R^{mn}),$ then the spatial gradient $D\bu$ of the weak
solution $\bu$ to \eqref{system} lies in $ L^p_\omega(Q,\R^{mn})$ and satisfies the estimate
$$
\|D\bu\|_{L^p_\omega(Q,\R^{mn})}\leq c
\|\bF\|_{L^p_\omega(Q,\R^{mn})},
$$
with a constant $c$ independent of $\bu$ and $\bF.$
\end{thm}
\begin{cor}
Under the assumptions of Theorem~\ref{th3}, there is a small positive constant
$\delta=\delta(n,m,L,\nu,p,\lambda)$ such that if the couple $(\bA,\Omega)$ is $(\delta,R)$-vanishing of codimension 1 and
$\bF\in
L^{p,\lambda}(Q, \R^{mn})$ with $p\in(2,\infty)$ and $\lambda\in(0,n+2),$
then $D\bu\in  L^{p,\lambda}(Q, \R^{mn})$ and
$$
\|D\bu\|_{  L^{p,\lambda}(Q, \R^{mn})}\leq c \| \bF\|_{
L^{p,\lambda}(Q, \R^{mn})}
$$
with a constant  $c$ independent of $\bu$ and $\bF.$
\end{cor}

\end{document}